\documentclass[a4paper,11pt,reqno]{amsart}
\usepackage[T1]{fontenc}
\usepackage[utf8]{inputenc}
\usepackage{lmodern}
\usepackage{amsmath, amsthm, amssymb,amscd, mathrsfs, amsfonts, mathtools}
\usepackage{hyperref}
\usepackage{euler}
\usepackage{times}
\usepackage[all]{xy}
\usepackage{todonotes}
\usepackage{xcolor}
\usepackage{tikz}
\usepackage{tikz-cd}
\usepackage{stmaryrd}
\usepackage[lite]{amsrefs}

\renewcommand{\PrintDOI}[1]{\href{http://dx.doi.org/\detokenize{#1}}{doi: \detokenize{#1}}%
  \IfEmptyBibField{pages}{, (to appear in print)}{}}

\def\commutatif{\ar@{}[rd]|{\circlearrowleft}}

\newcommand{\eq}[1][r]
   {\ar@<-3pt>@{-}[#1]
    \ar@<-1pt>@{}[#1]|<{}="gauche"
    \ar@<+0pt>@{}[#1]|-{}="milieu"
    \ar@<+1pt>@{}[#1]|>{}="droite"
    \ar@/^2pt/@{-}"gauche";"milieu"
    \ar@/_2pt/@{-}"milieu";"droite"}

\def\dar[#1]{\ar@<2pt>[#1]\ar@<-2pt>[#1]}
  \entrymodifiers={!!<0pt,0.7ex>+} %%%%%%%%%%%%%%%%%

\newcommand{\bigon}[4][r]{% %%%%% bigons
    \ar@/^1pc/[#1]^{#2}_*=<0.3pt>{}="HAUT"
    \ar@/_1pc/[#1]_{#3}^*=<0.3pt>{}="BAS"
    \ar@{=>} "HAUT";"BAS" ^{#4}
  }

\newcommand{\bigons}[6][r]{  %%%%% Vertical composition of bigons
    \ar@/^2pc/[#1]^{#2}_*=<0.3pt>{}="HAUT"
    \ar@{}    [#1]     ^*=<0.3pt>{}="MILIEUHAUT"
                       _*=<0.3pt>{}="MILIEUBAS"
    \ar[#1]_(0.3){#3}                  
    \ar@/_2pc/[#1]_{#4}^*=<0.3pt>{}="BAS"
    \ar@{=>} "HAUT";"MILIEUHAUT" ^{#5}
    \ar@{=>} "MILIEUBAS";"BAS" ^{#6}
  }

\newtheorem{thm}{Theorem}[section]
\newtheorem{pro}[thm]{Proposition}
\newtheorem{lem}[thm]{Lemma}
\newtheorem{cor}[thm]{Corollary}

\theoremstyle{definition}
\newtheorem{df}[thm]{Definition}

\newtheorem*{theorem*}{Theorem}
\newtheorem*{pro*}{Proposition}
\theoremstyle{remark}
\newtheorem{rmk}[thm]{Remark}

\newtheorem{ex}[thm]{Example}

\allowdisplaybreaks

\newcommand\Hom{\operatorname{Hom}}

\let\bf\mathbf

\def\Aut{\operatorname{Aut}}

\title{On Schur Algebras and Derivations of Free Lie Algebras}

\author{Frederick Cohen}

\address{Department of Mathematics, University of Rochester, RochesterHylan Bldg, 140 Trustee Rd, Rochester, NY 14627, U.S.A.}
\email{cohen@rochester.edu}

\author{Mohamed Elhamdadi} 
\address{Department of Mathematics, 
University of South Florida, Tampa, FL 33620, U.S.A.} 
\email{emohamed@math.usf.edu} 

\author{Tao Jin} 
\address{PBC School of Finance, Tsinghua University, 43 Chengfu Road, Haidian District, Beijing 100083, China} 
\email{jint@pbcsf.tsinghua.edu.cn} 

\author{Minghui Liu} 
\address{Mathematics and Science Department, Florida College, Temple Terrace, FL 33617, U.S.A.} 
\email{LiuM@floridacollege.edu}

\begin{document}

\maketitle

\begin{abstract}
We investigate the action of Schur algebra on the Lie algebras of derivations of free Lie algebras and operad structures constructed from it. We also show that the Lie algebra of derivations is generated by quadratic derivations together
with the action of the Schur operad. Applications to certain subgroups of the automorphism group of a finitely generated free group are given as well. 
\end{abstract}

\tableofcontents

%TO DO LIST:

%CHECK FORMULAS ON TOP OF PAGE 9

%COHEN's Comments:
%One quick comment. I wonder whether there is something more going on with the Lie algebra of derivations say Der, and the action of the Schur algebras. It is striking to me that one just needs 2 generators.

%Why not state these results more prominently  in the Introduction as follows

%(but roughly and inaccurately).:

%Proposition: The Schur algebras assemble to give an operad denoted Schur.

%Theorem: The Lie algebra of derivations Der is a "module" over the operad Schur which is generated by 2 elements".

%Also, I did not see where you verified the formulas

%for the action of the symmetric groups, but I could have easily missed a lot ! 

%Also, I wonder whether/how Schur duality

%fits in the picture. \\

%sept 16: TOP OF PAGE 9.  DOUBLE CHECK THE ACTION OF SIGMAq.  Check TAO thesis where he talks about this.

\section{Introduction}
Let $F_n$ be the free group of rank $n$ and ${IA}_n$ be the subgroup of $\mathrm{Aut}(F_n)$ which induces the identity automorphism on $\mathbb{Z}^n$, the abelianization of $F_n$. These groups have been studied by Nielsen in 1917 \cite{Nielsen}, but their structures remained not well-understood. In 1997, Krsti\'{c} and McCool \cite{KM} proved that $IA_3$ is not finitely presentable and $H^2(IA_3)$ is not finitely generated. Some properties of the second rational
cohomology of $IA_n$ were studied by A. Pettet \cite{P} in 2005. In 2007, M. Bestvina, K.-U. Bux and D. Margalit proved some properties about a certain subgroup of $\mathrm{Out}\left(F_n\right)$, the outer automorphism
group of $F_n$, and its homology \cite{BBM}. In 2011, S. Galatius \cite{G} proved that there is a homology equivalence
\[\mathbb{Z}\times B\left(\underset{{n\to\infty}}{\mathrm{colim}}\,\mathrm{Aut}(F_n)\right)\to\underset{{n\to\infty}}{\mathrm{colim}}\,\Omega^n S^n,\]
where the right hand side is the infinite loop space, and $B$ stands for the classifying space.

In order to study the groups $IA_n$ and $\mathrm{Aut}(F_n)$, the Johnson filtration 
\[\mathrm{Aut}(F_n)=\mathcal{A}_n(0) \supset \mathcal{A}_n(1) \supset\cdots \]

was first introduced by S. Andreadakis in 1965 \cite{A}. The Johnson Lie algebra $\mathrm{gr}_\ast^J(IA_n)$ is the Lie algebra structure associated with the Johnson filtration of the group $IA_n$. The group $IA_n$ is related to $\mathrm{Der}_\ast^\triangle (L_{\mathbb{Z}}(X_n))$, the Lie algebra of derivations of
the free Lie algebra $L_{\mathbb{Z}}(X_n)$, via the so-called Johnson homomorphism, which was
studied by N. Kawazumi \cite{K}, where $X_n=\{1,\ldots,n\}$. There is a second descending filtration $\mathcal{A}_n(k)'$ of $\mathrm{Aut}(F(X_n))$ induced by the lower central series of $IA_n$. S. Andreadakis conjectured that $\mathcal{A}_n(k) =
\mathcal{A}_n(k)'$ for all $k$ and $n$; which is shown to be false in general by L. Bartholdi in \cite{B0,B1}, though some special cases were studied and proved by S. Andreadakis \cite{A} and some other authors such as K. Satoh \cite{Sa}.

In this article which is based on the thesis of T. Jin \cite{Jin}, we introduce a new additional structure for the Lie algebra of derivations of a free Lie algebra, which arises from actions of a classical algebra known as the Schur algebra,
which was first studied by I. Schur in the early twentieth century to explore the
representation theory of the symmetric group $\Sigma_r$ and the general linear group
$\mathrm{GL}(m,\mathbb{C})$ \cite{Sch}. 
We also show that the Lie
algebra of derivations is generated by quadratic derivations together
with the action of the Schur operad. Applications to certain subgroups of the automorphism group of a finitely generated free group are given as well. The first main result of this paper is the following proposition:
\begin{pro*}
 The Schur algebras assemble to give an operad, denoted $\mathit{Schur}$.
\end{pro*}
See Proposition \ref{1.1.10} for details. The Lie algebra of derivations is a ``module'' over the operad $\mathit{Schur}$, which gives us the second main result of this paper (See Theorem \ref{1.1.13} for details):
 %More precisely, we have the following (Proposition \ref{1.1.10}):\\

		%Let $\mathit{Schur}=\left(\mathrm{Schur}(m-1,n,K)\right)_{m\in\mathbb{Z}_+}$. For each $m,k_1,\ldots,k_m\in\mathbb{Z}_{+}$, define the operad composition by 
		%\begin{align*}
		%\circ: P(m)\times P(k_1)\times P(k_m)&\to P(k_1+\cdots k_m)\\
		%(\theta,\theta_1,\ldots, \theta_m)&\mapsto \theta\boxtimes\theta_1\boxtimes\cdots\boxtimes\theta_m
		%\end{align*}
		%and the identity of $\mathscr{S}chur$ is defined as the identity element of $\mathrm{Schur}(0,n,K)=K=P(1)$. This defines $\mathit{Schur}$ as an operad.
\begin{theorem*}
 The Lie algebra of derivations $\mathrm{Der}_*=\displaystyle\bigoplus_{p=2}^\infty\Hom_K\left(V_n,L^p(X_n)\right)$ is a ``module'' over the operad $\mathit{Schur}$ which is generated by $2$ elements.
\end{theorem*}

%The Schur Operad $\mathscr{S}ch$ acts on $T\left(V_n(K)\right)$, $L_K(X_n)$, $\mathrm{Der}_*^{\triangle}\left(L_K(X_n)\right)$ and $\mathrm{Der}_*\left(L_K(X_n)\right)$. Furthermore, $\mathrm{Der}_*\left(L_K(X_n)\right)$ is generated by  $\{{v}_1,\ldots,{v}_n\} \subseteq {M_n}$ where either each ${v}_i$ is ${\chi_{i,j}}$ with $i\neq j$ or each ${v}_i$ is ${\theta}_{i,[x_s,x_t]}$ for distinct $i,s$ and $t$. 

The paper is mainly based on the thesis of Jin Tao \cite{Jin}.  The article is organized as follows: In Section \ref{Preliminaries}, we recall the necessary ingredients needed for the rest of the paper. In particular, the Lie algebra associated to the lower central series of a group $G$ and the group of so-called $\mathrm{IA}$-automorphisms of $G$. Section \ref{Calculation} deals with the calculation involving the derivations of a free Lie algebra in order to explore the structure of certain subgroups of the automorphism group of a finitely generated free group. We prove the main theorem stating that any two distinct elements in $\mathrm{IA}_n$ either generate a free group or a free abelian group. In Section \ref{Schur}, we recall the definition of Schur algebra and we prove a few results regarding the Schur algebra on the tensor product of a finitely generated free $K$-module.
In Section \ref{Associativity}, we study some properties of Schur algebras and their action on
the Lie algebra of derivations of a free Lie algebra. By using the inter-degree multiplication
induced from a transfer-like map, it is shown that the Schur algebra forms an associative and anti-commutative graded K-algebra. In Section  \ref{Action}, the Schur-operad structure is constructed.  We show that the Lie algebra of derivations can be generated by quadratic derivations together with the action of the Schur operad.
We introduce a short exact sequence of Lie algebras which is obtained from certain group filtrations in Section \ref{Exact}.

\section{Preliminaries}\label{Preliminaries}
In this section, we recall the necessary ingredients needed later in this paper including the Lie algebra associated to the lower central series of a group $G$ and the group of $\mathrm{IA}$-automorphism of $G$. Let $X_n=\{1,\ldots,n\}$ and let the free Lie algebra generated by $X_n$ over a commutative ring with identity $K$ be denoted by $L_K(X_n)$, or simply $L(X_n)$. The free Lie algebra $L_K(X_n)$ is naturally a graded $K$-module and 
\[L_K(X_n)=\bigoplus_{p=1}^{\infty}L_K^p(X_n),\] where each component $L_K^p(X_n)$ is a free $K$- module. Let $G$ be a group and $\{\Gamma^mG\}_{m\in\mathbb{Z}^+}$ be the lower central series of $G$. There is a naturally induced homomorphism
\[\mathrm{Aut}(G)\to\mathrm{Aut}(G/\Gamma^mG)\] for every positive integer $m$. Let the kernel of this induced homomorphism be denoted by $J^m\mathrm{Aut}(G)$. 
\begin{df}
A group $G$ is called residually nilpotent if 
$\displaystyle\bigcap_{m=1}^{\infty} \Gamma^mG$ is the trivial group.
\end{df}
Recall that for an algebra $A$ over a field $K$, a derivation $D:A\to A$ is a $K$-linear map such that
 \[D(xy)=(Dx)y+x(Dy).\] The set $\mathrm{Der}(A)$ of all derivations is a Lie algebra with the product 
 \[[D,D']=DD'-D'D.\]  
The derivations of the free Lie algebra $L(X_n)$ is a naturally graded $K$-module
\[\mathrm{Der}_{\ast}^{\triangle}\left(L(X_n)\right)=\bigoplus_{p=1}^{\infty}\mathrm{Der}_p\left(L(X_n)\right),\]
where $\mathrm{Der}_p\left(L(X_n\right))$ is the component of degree $2p$; that is,
\[\mathrm{Der}_p\left(L(X_n)\right)=\{f\in\mathrm{Der}_{\ast}^{\triangle}\left(L(X_n)\right)\mid f(x_i)\in L^p(X_n), 1\leq i\leq n \}.\]

\begin{rmk}\label{doubling}
	Notice that for notation in this paper, we adopt the convention that an element in $V_n(K)^{\otimes q}$ has degree $2q$; that is, we define
\[T\left(V_n(K)\right)=\bigoplus_{q=0}^\infty T^q\left(V_n(K)\right),\]
where 
\[T^{2q}\left(V_n(K)\right)=V_n(K)^{\otimes q}\]
and 
\[T^{2q+1}\left(V_n(K)\right)=0.\]In particular, the grading of   $Der_p (L(X_n))$ is the component of degree $2p$.
\end{rmk}

Since every $f\in\mathrm{Der}_{\ast}^{\triangle}\left(L(X_n)\right)$ is uniquely determined by the restriction of $f$ on $X_n$, we have
\[\mathrm{Der}_p\left(L(X_n)\right)\cong\Hom_K\left(V_n,L^p(X_n)\right)\]
and
\[\mathrm{Der}_{\ast}^{\triangle}\left(L(X_n)\right)\cong\bigoplus_{p=1}^\infty\Hom_K\left(V_n,L^p(X_n)\right).\]

The sub-Lie algebra $\mathrm{Der}_{\ast}\left(L(X_n)\right)$ which is defined as
\[\mathrm{Der}_{\ast}\left(L(X_n)\right)=\bigoplus_{p=2}^\infty\Hom_K\left(V_n,L^p(X_n)\right)\]is called the derivations of the free Lie algebra $L(X_n)$.

Let $V_n(K)$ (or simply $V_n$) be the free $K$-module generated by $X_n=\{x_1,\cdots,x_n\}$ and let $\Sigma_q$ be the symmetric group on $q$ letters. The symmetric group $\Sigma_q$ acts on $V_n(K)^{\otimes q}$ from the right by 
\[(y_1\otimes\cdots\otimes y_q)\cdot\sigma=y_{\sigma^{-1}(1)}\otimes\cdots\otimes y_{\sigma^{-1}(q)},\]
where $\sigma\in\Sigma_q$ and $y_1\otimes\cdots\otimes y_q\in V_n^{\otimes q}$.

Recall that for a given group $G$, an integral filtration of $G$ is a sequence $\{F^mG\}_{m\in\mathbb{Z}^+}$ of subgroups of $G$ such that (1) $F^1G=G$ and (2) $[F^mG,F^nG]\subseteq F^{m+n}G$. Given an integral filtration $\{F^mG\}_{m\in\mathbb{Z}^+}$, there is an associated graded Lie algebra structure
\[\mathrm{gr}_{\ast}G=\bigoplus_{m=1}^{\infty}\mathrm{gr}_mG,\]
where $\mathrm{gr}_\ast G=F^mG/F^{m+1}G$ \cite{S}. 

	Let $G$ be a group and $F^{m}G=\Gamma^m G$ be the $m$-th term of the lower central series of $G$. We then have the following definition, which is a special case of associated graded Lie algebra:
\begin{df}
 The associated graded Lie algebra is called the Lie algebra associated to the lower central series of $G$ and is denoted by $\mathrm{gr}^{\mathrm{LCS}}_\ast G$.
\end{df}
\begin{rmk}
the Lie algebra associated to the lower central series of $G$ is a Lie algebra, and graded, but it is not a "graded Lie algebra" as defined by Milnor-Moore \cite{MM}. For example, in a graded Lie algebra it may be the case that $[x,x]$ is not equal to $0$. In the graded case 
$$[x,y] =(-1)^{|x||y|} [y,x].$$ Thus if $x = y$ are of degree 1, then $$[x,x] = 2x^2$$ in the universal enveloping algebra (over a field).

To avoid this issue, we may double the grading of each element in the Lie algebra
for the descending central series where $[x,x] = 0$.
\end{rmk}
\begin{lem}\label{3.1.8}
	Let $f$ be a homomorphism from a residually nilpotent group $G$ to a group $H$. If the induced Lie algebra homomorphism
	\[f_{\ast}:\mathrm{gr}^{\mathrm{LCS}}_{\ast}G\to\mathrm{gr}^{\mathrm{LCS}}_{\ast}H\] is a monomorphism, then so is $f$.
\end{lem}
\begin{proof}
Suppose $f$ is not a monomorphism, then there exists some non-identity element $g\in G$ such that $f(g)=1$. Since $G$ is residually nilpotent, $g\in\Gamma^r G\setminus\Gamma^{r+1}G$ for some positive integer $r$. So
\[f\left(g\cdot\Gamma^{r+1}G\right)\subseteq \Gamma^{r+1}H.\]
But since $f_\ast$ is one-to-one, this implies $g\cdot\Gamma^{r+1}G$ is the identity element in $\Gamma^{r}G/\Gamma^{r+1}G$, which implies $g\in\Gamma_{r+1}G$ and this is a contradiction.   
\end{proof}
\begin{df}
Let $G$ be a group. The group of IA-automorphism of $G$ is defined to be \[J^2\mathrm{Aut}(G):=\mathrm{ker}\left(\mathrm{Aut}(G)\to\mathrm{Aut}\left(G/[G,G]\right)\right).\] That is, each IA-automorphism of $G$ is defined as an automorphism of $G$ which induces the identity automorphism on $G/[G,G]$. Furthermore, we write $J^2\mathrm{Aut}\left(F(X_n)\right)$ as $\mathrm{IA}_n$. 
\end{df}

Let $J^m\mathrm{Aut}(F(X_n))$ be the group of automorphisms of $F(X_n)$ which induces identity on the quotient group $F(X_n)/\Gamma_{m+1}F(X_n)$. The Johnson filtration on $IA_n$ is defined as the filtration
\[IA_n=J^2\mathrm{Aut}(F(X_n))\supseteq J^3\mathrm{Aut}(F(X_n))\supseteq\cdots.\]

Clearly the kernel of the natural quotient
\[\Aut(F(X_n))\to \mathrm{GL}(n,\mathbb{Z})\] is precisely $IA_n$. A finite set of generators of $IA_n$ is given in \cite{MR2109550} as follows:
\[M_n=\{\chi_{i,j}\mid 1\leq i,j\leq n, i\neq j\}\bigcup\{\theta_{i,[x_s,x_t]}\mid i\notin \{s,t\},1\leq i\leq n, 1\leq s<t\leq n\},\]where
\[
\chi_{i,j}(x_k)=\begin{cases}
x_k,\quad\quad\text{if } k\neq i,\\
x_j^{-1}x_kx_j, \text{if } k=i.
\end{cases}
\]and
\[\theta_{i,[x_s,x_t]}(x_k)=\begin{cases}
x_k,\quad\quad\text{if } k\neq i,\\
x_i[x_s,x_t], \text{if } k=i.
\end{cases}\]  

We define two types of derivations of $L(X_n)$ as follows:

\[\widetilde{M}_n=\{\widetilde{\chi}_{i,j}\mid 1\leq i,j\leq n, i\neq j\}\bigcup\{\widetilde{\theta}_{i,[x_s,x_t]}\mid i\notin \{s,t\},1\leq i\leq n, 1\leq s<t\leq n\},\]where
\[
\widetilde{\chi}_{i,j}(x_k)=\begin{cases}
0,\quad\quad\text{if } k\neq i,\\
[x_i,x_j], \text{if } k=i.
\end{cases}
\]and
\[\widetilde{\theta}_{i,[x_s,x_t]}(x_k)=\begin{cases}
0,\quad\quad\text{if } k\neq i,\\
[x_s,x_t], \text{if } k=i.
\end{cases}\]  

\section{Calculation with the Derivations}\label{Calculation}
This section deals with the calculation involving the derivations of a free Lie algebra in order to explore the structure of certain subgroups of the automorphism group of a finitely generated free group. We prove the main theorem stating that any two distinct elements in $\mathrm{IA}_n$ either generate a free group or a free abelian group. Let $u=[\widetilde{\chi}_{i,j_1},\ldots,\widetilde{\chi}_{i,j_r}]$ be a Lie monomial of degree $2r$ in $\mathrm{Der}_{\ast}\left(L_{\mathbb{Z}}(X_n)\right)$ ($n\geq 3$) and $\check{u}$ be the Lie monomial obtained from $u$ by replacing $\widetilde{\chi}_{i,j_k}$ in $u$ with $x_{j_k}$ for every $k=1,\ldots,r$.

The following lemma can be easily proved by induction on the degree of simple Lie monomials:
\begin{lem}\label{3.2.2}
	We have
	\[u(x_k)=\begin{cases}
	[x_i,\check{u}],\text{ if } k=i;\\0,\text{ otherwise.}
	\end{cases}\]
\end{lem}
The following proposition can be found in \cite{CP}:
\begin{pro}\label{3.1.6}
	Let $L[S]$ be the free Lie algebra over the integers generated by a set $S$, and let $a$ be an element of $S$ with $S$ of cardinality at least $2$. Then the centralizer of $a$ in $L[S]$ is the linear span of $a$.
\end{pro}

\begin{lem}\label{3.2.3}
	The set $\widetilde{C}_r:=\{\widetilde{\chi}_{i,j_1},\ldots,\widetilde{\chi}_{i,j_r}\}$, $1\leq i\leq n$ generates a free sub-Lie algebra $S(\widetilde{C}_r,\mathbb{Z})$ in $\mathrm{Der}_{\ast}\left(L_{\mathbb{Z}}(X_n)\right)$.
\end{lem}
\begin{proof}
	Let $Y_r=\{y_1,\ldots,y_r\}$ be a set of $r$ elements and let $\widetilde{f}_r:L_{\mathbb{Z}}(Y_r)\to S(\widetilde{C}_r,\mathbb{Z})$ be the Lie epimorphism extended from the map $f_r$ which maps every $y_k$ to $\widetilde{\chi}_{i,j_k}$.
	
	Let $u=\sum a_ku_k$ be a non-zero element in $L_{\mathbb{Z}}(Y_r)$, where each $u_k$ is a Lie monomial. For every Lie monomial $v$, let $v^{\#}$ be the lie monomial in $\mathrm{Der}_{\ast}\left(L_{\mathbb{Z}}(X_n)\right)$ obtained by replacing every $y_k$ by $\widetilde{\chi}_{i,j_k}$. By Lemma \ref{3.2.2}, 
	\[\widetilde{f}_r(u)(x_k)=\sum_{k}a^ku_k^{\#}(x_k)=\begin{cases}
	[x_i,\sum_ka_k\check{u}^{\#}], $ if $k=i\\
	0, $ if $ k\neq i.
	\end{cases}\]
	Since $\sum_ka_k\check{u}^{\#}$ is non-zero, so is $[x_i,\sum_ka_k\check{u}^{\#}]$ by Proposition $\ref{3.1.6}$. Thus $\widetilde{f}_r(u)$ is non-zero and therefore, $\widetilde{f}_r$ is an isomorphism.
\end{proof}

Let $N$ be a fully invariant subgroup of $F(X_n)$. The quotient group $F(X_n)/N$ is called the \emph{universal variety group} associated to $N$.

\begin{df}Let $G$ be a universal variety group, then the Lie algebra associated with the Johnson filtration on $\mathrm{IA}_n(G)$ is called the Johnson Lie algebra associated with group $G$ and is denoted by $\mathrm{gr}_{\ast}^{\mathrm{J}}\left(\mathrm{IA}_n(G)\right)$.
\end{df}

	Let $A$ be a subgroup of a group $B$ and an integral filtration $\{B_m\}$ naturally induces a filtration $\{A_m\}$ on $A$. One easily proves the following lemma:
\begin{lem}\label{3.2.4}
The inclusion $i:A\hookrightarrow B$ induces a monomorphism of Lie algebras 
\[i_{\ast}:\mathrm{gr}_{\ast}A\to\mathrm{gr}_{\ast}B.\]
\end{lem}
As a special case, for a subgroup $G$ of $\mathrm{IA}_n$, the inclusion $i_G:G\hookrightarrow \mathrm{IA}_n$ induces a monomorphism of Lie algebra 
\[i_{G,\ast}:\mathrm{gr}_{\ast}^J(G)\to \mathrm{gr}_{\ast}^J(\mathrm{IA}_n).\]

The following lemma follows from Lemma \ref{3.2.4} and the fact that $\widetilde{J}_{n,\ast}$ is a monomorphism of Lie algebras:
\begin{lem}\label{3.2.6}
	For every subgroup $G$ of $\mathrm{IA}_n$, $\widetilde{J}_{G,\ast}$ is a monomorphism of Lie algebras.
\end{lem}

The following lemma is proven in \cite{S}:
\begin{lem}\label{2.3.6}
Let $F(X)$ be the free group generated by a set $X$. The canonical map $X\to \mathrm{gr}_1^{\mathrm{LCS}}\left(F(X)\right)=F(X)/\left[F(X),F(X)\right]$ induces a natural isomorphism
\[\phi_{X,\ast}:L_{\mathbb{Z}}(X)\to\mathrm{gr}_\ast^{\mathrm{LCS}}\left(F(X)\right).\]
\end{lem}

The following lemma given in \cite{Jin} will be used to prove Theorem \ref{1.1.14}:
\begin{lem}\label{3.2.7}
Let $n\geq r+1$ and let $G$ be the subgroup of $\mathrm{IA}_n$ generated by $C_r:=\{\chi_{i,j_1},\ldots,\chi_{i,j_r}\}$, then $G$ is a free group and the natural homomorphism	
\[\psi_{G,\ast}:\mathrm{gr}_\ast^{\mathrm{LCS}}(G)\to\mathrm{gr}_\ast^{\mathrm{J}}(G)\] is a monomorphism of Lie algebras. Let
\[\widetilde{J}_{G,\ast}:\mathrm{gr}_\ast^{\mathrm{J}}(G)\to\mathrm{Der}_{\ast}\left(L_{\mathbb{Z}}(X_n)\right)\] be the Johnson homomorphism of $G$. Then $\widetilde{J}_{G,\ast}\left(\psi_{G,\ast}\left(\mathrm{gr}_\ast^{\mathrm{LCS}}(G)\right)\right)$ is exactly the free sub-Lie algebra $S(\widetilde{C}_r,\mathbb{Z})$ in $\mathrm{Der}_\ast\left(L_{\mathbb{Z}}(X_n)\right)$.
\end{lem}
Let $L_n$ be a collection of $n$ unknotted, unlinked circles in $\mathbb{R}^3$, and let $P\Sigma_n$ be the group of motions of $L_n$ where each circle ends up back at its original position. The group $P\Sigma_n$ can be represented as a group of free group automorphisms; that is, $P\Sigma_n$ is the subgroup of $\mathrm{Aut}\left(F(X_n)\right)$ defined by
\[P\Sigma_n:=\{\alpha\in\mathrm{Aut}\left(F(X_n)\right)\mid \alpha(x_i)\text{ is conjugate to }x_i, i=1,\ldots,n\};\]
see \cite{JMM} for details.

A presentation of $P\Sigma_n$ given in the following theorem appeared in \cite{McCool}:

\begin{thm}\label{5.2.1}
A presentation of $P\Sigma_n$ is given by generators $\chi_{k,j} (k\neq j)$ with the following relations.
\begin{enumerate}
	\item For distinct $i,j,k$, $\chi_{i,j}\cdot\chi_{k,j}\cdot\chi_{i,k}=\chi_{i,k}\cdot\chi_{i,j}\cdot\chi_{k,j}$.
	\item If $\{j,k\}\cap\{s,t\}=\emptyset$, then $[\chi_{k,j},\chi_{s,t}]=1$,
	\item $[\chi_{k,j},\chi_{s,j}]=1$,
	\item For distinct $i,j,k$, $[\chi_{i,k},\chi_{i,j}]=[\chi_{i,k},\chi_{k,j}^{-1}]$.
\end{enumerate}
\end{thm}
The following lemma can be proved by induction on the degree of Lie monomials using Jacobi identity, while the base case is clear from Theorem \ref{5.2.1}:
\begin{lem}\label{3.2.1}
Let $u$ be a Lie monomial consisting of $\widetilde{\chi}_{i,j}$ and $\widetilde{\chi}_{i,j'}$ for distinct $i, j, j'$ in $\mathrm{Der}_\ast\left(L_{\mathbb{Z}}(X_n)\right)$ $(n\geq 3)$ and $\hat{u}$ be the Lie monomial obtained by replacing all $\widetilde{\chi}_{i,j'}$ with $-\widetilde{\chi}_{j,j'}$ in the expression of $u$. Then when $\mathrm{deg}(u)\geq 2$,
\[u=\hat{u}.\]
\end{lem}
\begin{thm}\label{1.1.14}
Any two distinct elements in ${IA}_n$ either generate a free group or a free abelian group. More precisely, let $\chi_{i,j}$ and $\chi_{i',j'}$ be two distinct elements in ${IA}_n$. Then $\chi_{i,j}$ and $\chi_{i',j'}$ generate
\begin{enumerate}
	\item a free abelian group, when $\{i,j\}\cap\{i',j'\}=\emptyset$, or  $j=j'$ and $i\neq i'$.
	\item a free group, otherwise.
\end{enumerate}
\end{thm}
\begin{proof}
When $\{i,j\}\cap\{i',j'\}=\emptyset$, or  $j=j'$ and $i\neq i'$, the conclusion is straightforward. Otherwise, there are a few possibilities:
\begin{enumerate}
	\item When $i=j'$ and $j=i'$, this case is essentially the same as $\chi_{1,2}$ and $\chi_{2,1}$ in $IA_2$. Nielsen \cite{Nielsen} proved that $IA_2$ is a free group generated by $\chi_{1,2}$ and $\chi_{2,1}$ and the result follows.
	\item When $i=i'$ and $j\neq j'$. When $i>\mathrm{max}\{j,j'\}$, it can be proved using upper triangular McCool group \cite{CPVW}.

	\item When $j=i'$ and $i\neq j'$, let $H$ and $G$ be the subgroups of $IA_n$ generated by $\{\chi_{i,j},\chi_{j,j'}\}$ and $\{\chi_{i,j},\chi_{i,j'}\}$, respectively. Let $f:G\to H$ be the group epimorphism defined by $f(\chi_{i,j})=\chi_{i,j}$ and $f(\chi_{i,j'})=\chi_{j,j'}^{-1}$.
	
Consider the following diagram, which is commutative for $\ast\geq 2$ by Lemma \ref{3.2.1}, where $\psi_{G,\ast}$ and $\psi_{H,\ast}$ are the natural homomorphisms from the LCS Lie algebras of $G$ and $H$ to the Johnson Lie algebras of $G$ and $H$, respectively; $\widetilde{J}_{G,\ast}$ and $\widetilde{J}_{H,\ast}$ are the Johnson homomorphism of $G$ and $H$, respectively.
$$
\begin{tikzcd}
&	gr^{LCS}_\ast (H) \arrow[dr,"\psi_{H,\ast} "] &\\
gr^{LCS}_\ast (G)  \arrow[ur,"f_\ast^{LCS}"]\arrow[d,"\psi_{G,\ast} "] &    & gr^{J}_\ast (H)  \arrow[d,"\widetilde{J}_{H,\ast}"]\\
gr_\ast^J(G) \arrow[rr,"\widetilde{J}_{G,\ast}"] &     &  \mathrm{Der}\left(L_{\mathbb{Z}}(X_n)\right)  \\
%& DD \arrow[ur,swap," "]    &
\end{tikzcd}
$$

By the lemmas \ref{3.2.6} and \ref{3.2.7}, the map $\widetilde{J}_{G,\ast}\circ\psi_{G,\ast}$ is a monomorphism of Lie algebras, which implies that $f_\ast^{\mathrm{LCS}}$ is a monomorphism when $\ast\geq 2$ while the case $\ast =1$ is obvious.
\item When $i=j'$ and $j\neq i'$, the proof is essentially the same as the previous case.
\end{enumerate}
Since 
$f_\ast^{\mathrm{LCS}}:\mathrm{gr}_\ast^{\mathrm{LCS}}(G)\to\mathrm{gr}_\ast^{\mathrm{LCS}}(H)$ is a monomorphism, so is $f$ and the proof is finished.

\end{proof}

Let $G$ be a universal variety group and $m\geq 2$, there is a monomorphism
\[J^m\mathrm{Aut}(G)/J^{m+1}\mathrm{Aut}(G)\to\Hom_{\mathbb{Z}}(G/[G,G],\Gamma^m G/\Gamma_{m+1} G).\]

When $G=F(X_n)$, this monomorphism gives a $\mathbb{Z}$-linear injection
\[\widetilde{J}_{n,m}:\mathrm{gr}_m^{\mathrm{J}}({IA}_n)\to\mathrm{Der}_{m+1}(L_{\mathbb{Z}}(X_n)).\]  

Therefore, one obtains a $\mathbb{Z}$-module monomorphism, called the Johnson homomorphism of $\mathrm{IA}_n$:

\[\widetilde{J}_{n,\ast}:\mathrm{gr}_\ast^{\mathrm{J}}({IA}_n)\to\bigoplus_{m=2}^{\infty}\mathrm{Der}_m\left(L_\mathbb{Z}(X_n)\right).\]

\section{The Schur Algebra on $V_n(K)^{\otimes q}$} \label{Schur}
 In this section we recall the definition of Schur algebra and we prove a few results regarding the Schur algebra on the tensor product of a finitely generated free $K$-module.

\begin{df}\label{defschur}
	The \textit{Schur algebra} on $V_n^{\otimes q}$ is the subalgebra of the algebra of endomorphisms on $V_n^{\otimes q}$ consisting of those $K$-linear maps $f$ which make the following diagram commute for all $\sigma\in\Sigma_q$,
\[
\begin{tikzcd}
V_n ^{\otimes q} \arrow[r, "f"] \arrow[d, "\sigma"]& V_n^{\otimes q}  \arrow[d, "\sigma" ] \\
V_n^{\otimes q}  \arrow[r,  "f" ]& V_n^{\otimes q}
\end{tikzcd}
\]
We denote the Schur algebra on $V_n^{\otimes q}$ by $\mathrm{Schur}(q,n,K)$. By convention we define $V_n^{\otimes 0}=K$.
\end{df}

The Schur algebra on $T(V_n)$ will be denoted by 
$\displaystyle{Schur(n,K)= \bigoplus_{q=0}^{\infty}  \mathrm{Schur}(q,n,K)}$, where  $T(V_n)=  \displaystyle{\bigoplus_{q=0}^{\infty} V_n^{\otimes q} }$ is the tensor algebra. A natural basis of $  V_n^{\otimes q}$ is 
\[
B_{n,q}:=\{ x_{i_1} \otimes \cdots \otimes x_{i_q} , \quad 1 \leq i_1, \ldots, i_q \leq n \}.
\]
We consider the following proper subset $S_{n,q}$ of $B_{n,q}$ given by 
\[
S_{n,q}:=\{ x_{i_1} \otimes \cdots \otimes x_{i_q} , \quad 1 \leq i_1 \leq  \ldots \leq  i_q \leq n \}.
\]

Let $\Sigma_q$ be the symmetric group on $q$ letters and for the $\Sigma_q$-action on $B_{n,q}$, the corresponding left $\Sigma_q$-orbit is given by
\[\mathscr{O}_{n,q}:=\{\sigma\cdot u\mid \sigma\in \Sigma_q \text{ and } u\in B_{n,q}\}=\{\sigma\cdot u\mid \sigma\in \Sigma_q \text{ and } u\in S_{n,q}\}.\]

\begin{df}\label{2.2.14}
	For a Lie monomial $u\in L_{X_n}^p$, let $\mathrm{deg}_{x_i}(u)$ be the number of occurrence of the letter $x_i$ in the expression of $u$. The multi-degree of $u$ is defined as 
	\[\mathrm{mdeg}(u)=\left(\mathrm{deg}_{x_1}(u),\ldots,\mathrm{deg}_{x_n}(u)\right).\]
\end{df}

Define $T^{\triangle}\left(V_n(K)\right)$ as the two-sided ideal generated by $V_n(K)$; that is,
\[T^{\triangle}\left(V_n(K)\right)=\bigoplus_{q=1}^{\infty}T^q\left(V_n(K)\right).\]
Let \[\mathfrak{b}_n:T^{\triangle}\left(V_n(K)\right)\to L_K(X_n)\]by
\[\mathfrak{b}_{n,2p}(x_{i_1}\otimes\cdots\otimes x_{i_p})=\mathrm{ad}(x_{i_1})\cdots\mathrm{ad}(x_{i_{p-1}})(x_{i_p})\text{,}\]
where $\mathfrak{b}_{n,2p}$ is the degree $2p$ component of the map $\mathfrak{b}_n$.
\begin{rmk}
	The map $\mathfrak{b}_n$ is essentially the standard map defined by Specht and Wever; see page 169 of \cite{Jacobson} and page 1197 of \cite{Cohen95} for details. In particular,  the classical Specht-Wever relation can be written as
	
	$$b_{n,2p}\circ  b_{n,2p} =  pb_{n,2p}.$$
\end{rmk}
The proof of the following proposition is straightforward:
\begin{pro}\label{4.1.1}
Let $u$ and $v$ be elements in $B_{n,q}$. Then \[\{\sigma\cdot u\mid \sigma\in \Sigma_q \}=\{\sigma\cdot v\mid \sigma\in \Sigma_q \}\]
if and only if
\[\mathrm{mdeg}\left(\mathfrak{b}_{n,2q}(u)\right)=\mathrm{mdeg}\left(\mathfrak{b}_{n,2q}(v)\right).\]
\end{pro}

For $u \in B_{n,q}$, let $\mathcal{I}_u$ be the stabilizer subgroup of $u$ under the action of $\Sigma_q$; that is, \[\mathcal{I}_u=\{g\in\Sigma_q\mid g\cdot u=u\}.\]

\begin{thm}\label{thm 4.1.2}
For any $v \in B_{n,q} \setminus S_{n,q}$, there is a unique $u \in S_{n,q}$ and some $\sigma \in \Sigma_q$ such that $v=\sigma(u)$.  If $f \in Sch(q,n,K)$ then $f$ satisfies the following two conditions
\begin{enumerate}

\item
$f(v)=\sigma(f(u))$ for any $\sigma \in \Sigma_q$ satisfying $v=\sigma(u), v \in B_{n,q} \setminus S_{n,q}$ and $u \in S_{n,q}$;

\item

Let $f(u)=\displaystyle{ \sum_{w \in B_{n,q}} a_w w} $ for $u \in S_{n,q}$, %where $a_w$ is the
If $\mathcal{I}_u \cdot w= \mathcal{I}_u \cdot w'$, then $a_w=a_{w'}$. Or equivalently, the following is true for every $u\in S_{n,q}$:
\[f(u)=\sum_{\mathcal{I}_u\cdot w\in\mathcal{O}_u}\left(a_{\mathcal{I}_u\cdot w}\cdot\sum_{v\in B_{n,q}, \mathcal{I}_u\cdot v=\mathcal{I}_u\cdot w}v\right),\]where $a_{\mathcal{I}_u\cdot w}\in K$.
\end{enumerate}
Conversely, if $f\in\mathrm{End}\left(V_n^{\otimes q}\right)$ satisfies the above two conditions, then $f\in\mathrm{Schur}(q,n,K)$.
\end{thm}
\begin{proof}
When  $f \in Sch(q,n,K)$ then $f$, then $f$ satisfies the first condition by Definition \ref{defschur} and we shall prove the second condition. Assuming $\mathcal{I}_u \cdot w= \mathcal{I}_u \cdot w'$, let $\sigma$ be an element in $\mathcal{I}_u$ such that $\sigma\cdot w=w'$. Then
\begin{align*} 
\sum_{w \in B_{n,q}} a_w w&=  f(u)\\ 
	&=  f(\sigma \cdot u)\\
	&=\sigma\cdot f(u)\\
	&=\sigma\cdot \sum_{w \in B_{n,q}} a_w w\\
	&=\sum_{w \in B_{n,q}} a_w \sigma \cdot w
\end{align*}
Since $\sigma\cdot w=w'$, we conclude that $a_w=a_{w'}$.

Now we prove condition (1) and (2) imply $f\in \mathrm{Sch}(q,n,K)$. It suffices to prove 
\[f\left(\sigma(w)\right)=\sigma\left(f(w)\right)\] for every $w\in B_{n,q}$ and $\sigma\in\Sigma_q$.

\paragraph{Case 1} When $\sigma(w)=w$ and $w\in S_{n,q}$, condition (2) implies $f\left(\sigma(w)\right)=\sigma\left(f(w)\right)$.  We write $w=\sum_{\theta\in B_{n,q}}a_\theta \theta$. Therefore, $\sigma\cdot f(w)=\sum_{\theta\in B_{n,q}}a_\theta \sigma\cdot\theta$. 
	 
	Since $\mathcal{I}_w\cdot (\sigma\cdot\theta)=\mathcal{I}_w\cdot \theta$ (Notice if $h\in I_w$, so is $h\cdot \sigma$), $a_{\sigma\cdot\theta}=a_\theta$, so $\sigma\cdot f(w)=\sum_{\theta\in B_{n,q}}a_\theta \sigma\cdot\theta=\sum_{\theta\in B_{n,q}}a_\theta\theta$.
%Here we should not use condition 1 as $\sigma(w)$ is not in $B_{n,q}\setminus S_{n,q}$

\paragraph{Case 2}
When $\sigma(w)=w$ and $w\notin S_{n,q}$, condition (1)  implies $f\left(\sigma(w)\right)=\sigma\left(f(w)\right)$. 

\paragraph{Case 3}When $\sigma(w)\neq w$ and $w\in S_{n,q}$. In this case $\sigma\cdot w\in B_{n,q}\setminus S_{n,q}$ and the conclusion is clear by condition (1).
\paragraph{Case 4}When $\sigma(w)\neq w$ and $w\notin S_{n,q}$.
therefore exists a unique $u\in S_{n,q}$ and some $\sigma'\in\Sigma_q$ such that $w=\sigma'(u)$.

\begin{enumerate}
	\item When $\sigma(w)\in S_{n,q}$, then condition (1) implies
	\[f(w)=\sigma^{-1}(f(\sigma(w))),\] and the result follows.
	\item  When $\sigma(w)\in B_{n,q}\setminus S_{n,q}$, then there are $u\in S_{n,q}$ and $\sigma'\in\Sigma_q$ such that $w=\sigma'(u)$. we obtain from condition (1) that 
	$f(w)=\sigma'\left(f(u)\right)$
	and a direct calculation shows that
	$f\left(\sigma(w)\right)=f(\sigma(\sigma'(u)))=\sigma(\sigma'(f(u)))\sigma\left(f(w)\right)$.
\end{enumerate} 
\end{proof}
The following theorem gives a necessary and sufficient condition when a map from $S_{n,q}$ to $V_n^{\otimes q}$ can be extended to a $\Sigma_q$-linear map.

\begin{thm}\label{thm 4.1.3}
A map $g$ from $S_{n,q}$ to $V_n^{\otimes q}$ can be extended to a $\Sigma_q$-linear map if and only if $g$ satisfies condition (2) in Theorem \ref{thm 4.1.2}.
\end{thm} 
\begin{proof}
	By Theorem \ref{thm 4.1.2}, we only need to prove the case when $g$ satisfies condition (2) in Theorem \ref{thm 4.1.2}, it can be extended to a $\Sigma_q$-linear map.
	
We extend $g$ to $B_{n,q}\setminus S_{n,q}$ in the following manner: for each $v\in B_{n,q}\setminus S_{n,q}$, there exists a unique $u\in S_{n,q}$ and some $\sigma\in\Sigma_q$ such that $v=\sigma(u)$. Define $f(v)=\sigma\left(g(u)\right)$. Condition (2) in Theorem \ref{thm 4.1.2} guarantees that $f$ is well-defined. When $u\in S_{n,q}$, define $f(u)=g(u)$. So we have a well-defined map $f:B_{n,q}\to B_{n,q}$ which extends $g$. Therefore we have an extension (still denoted by $f$) of $g$ from $V_n^{\otimes q}$ to itself. By the construction of $f$, it satisfies the two conditions in Theorem \ref{thm 4.1.2} and is therefore $\Sigma_q$-linear.
\end{proof}
When the extension Theorem \ref{thm 4.1.3} exists, it is clearly unique. The following corollary is proved in essentially the same manner.
\begin{cor}\label{cor 4.1.4}
A function from $R_{n,q}$, a set of representatives of the $\Sigma_q$-orbits in $B_{n,q}$, to $V_n^{\otimes q}$ can be extended to a $\Sigma_q$-linear map if and only if for every $u\in R_{n,q}$,\[g(u)=\sum_{\mathcal{I}_u\cdot w\in\mathcal{O}_u}\left(a_{\mathcal{I}_u\cdot w}\cdot\sum_{v\in B_{n,q}, \mathcal{I}_u\cdot v=\mathcal{I}_u\cdot w }v\right),\]where $a_{\mathcal{I}_u\cdot w}\in K$. 
\end{cor}

Consider the canonical embedding
\[I_n:L(X_n)\to T(V_n)\] and its restriction \[I_{n,q}:L^q(X_n)\to T^q(V_n).\] 
\begin{df}\label{def 4.1.8}
	Let $\mathrm{Br}_q$ be the set of Lie bracketing of $Lie$ monomials of degree $2q$, where the degree is doubled as in Remark \ref{doubling}. The the embedding $I_{n,q}$ of $L^q(X_n)$ into $T^q(V_n)$ induces a mapping, called the $q$-th bracketing function, $\mathrm{br}_q:\mathrm{Br}_q\to K(\Sigma_q)$ such that for every Lie monomial $[x_{i_1},\cdots, x_{i_q}]$,
	\[I_{n,q}\left([x_{i_1},\cdots,x_{i_q}]\right)=\mathrm{br}_q\left(\llbracket [x_{i_1},\cdots,x_{i_q}]\rrbracket\right)(x_{i_1}\otimes\cdots\otimes x_{i_q}).\]
\end{df}

\begin{ex}\label{ex 4.1.7}
\begin{align*}
I_{n,3}\left(\left[[x_i,x_j],x_k\right]\right)=&[x_i\otimes x_j-x_j\otimes x_i,x_k]\\
=&(x_i\otimes x_j-x_j\otimes x_i)\otimes x_k-x_k\otimes (x_i\otimes x_j-x_j\otimes x_i)\\
=&x_i\otimes x_j\otimes x_k-x_j\otimes x_i\otimes x_k-x_k\otimes x_i\otimes x_j+x_k\otimes x_j\otimes x_i\\
=&1(x_i\otimes x_j\otimes x_k)-(12)(x_i\otimes x_j\otimes x_k)\\
&-(123)(x_i\otimes x_j\otimes x_k)+(13)(x_i\otimes x_j\otimes x_k)\\
=&\left(1-(12)-(123)+(13)\right)(x_i\otimes x_j\otimes x_k)
\end{align*}
Therefore, 
\[\mathrm{br}_3\left(\left[[,],\right]\right)=1-(12)-(123)+(13)\in K(\sigma_3).\]
Similarly, 
\[\mathrm{br}_3\left(\left[,[,]\right]\right)=1-(23)-(132)+(13)\in K(\sigma_3).\]
\end{ex}

Every element $f\in\mathrm{Schur}(q,n,K)$ sends Lie elements to Lie elements, as is shown in the following theorem.
\begin{thm}\label{4.1.10}
	For any $f\in\mathrm{Schur}(q,n,K)$ and Lie monomial $[x_{i_1},\cdot,x_{i_q}]$, we have 
\[f\left([x_{i_1},\cdot,x_{i_q}]\right)=\sum_{\mathcal{I}_u\cdot w\in\mathcal{O}_u}\left(a_{\mathcal{I}_u\cdot w}\cdot\sum_{v\in B_{n,q}, \mathcal{I}_u\cdot v=\mathcal{I}_u\cdot w}[v]\right),\]
where $u=x_{i_1}\otimes\cdots\otimes x_{i_q}$, $a_{\mathcal{I}_u\cdot w}\in K$ and all the Lie monomials have the same bracketing. 

Let the cardinality of $\mathcal{O}_u$ be denoted by $d_u$. Fix $u$ and let $f$ run through $\mathrm{Schur}(q,n,K)$, then the $d_u$-tuple $(a_{\mathrm{I}_u\cdot w})_{\mathrm{I}_u\cdot w\in\mathcal{O}_u}$	runs through $K^{d_u}$.
\end{thm}
\begin{proof}
	From the definition of the $q$-th bracketing function $\mathrm{br}_q$ and the $\Sigma_q$ linearity of $f$, we conclude that
	\[f\left([x_{i_1},\cdots,x_{i_q}]\right)=\mathrm{br}_q\left(\llbracket [x_{i_1},\cdots,x_{i_q}]\rrbracket\right)\left(f(u)\right).\]
	
Since the orbits of $\mathcal{I}_u$'s action on $B_{n,q}$ form a partition of $B_{n,q}$, we may write 
\[f(u)=\sum_{v\in B_{n,q}}a_v\cdot v=\sum_{\mathcal{I}_u\cdot w\in\mathcal{O}_u}\left(\sum_{v\in B_{n,q}, \mathcal{I}_u\cdot v=\mathcal{I}_u\cdot w }a_v\cdot v\right).\] By the second condition in \ref{thm 4.1.2}, we conclude that 
\[f\left([x_{i_1},\cdots,x_{i_q}]\right)=\sum_{\mathcal{I}_u\cdot w\in\mathcal{O}_u}\left(a_{\mathcal{I}_u\cdot w}\sum_{v\in B_{n,q}, \mathcal{I}_u\cdot v=\mathcal{I}_u\cdot w }\mathrm{br}_q\left(\llbracket[x_{i_1},\cdots,x_{i_q}]\rrbracket\right)(v)\right).\]	
From the definition of $\mathrm{br}_q$, 
\[\mathrm{br}_q\left(\llbracket[x_{i_1},\cdots,x_{i_q}]\rrbracket\right)(v)=[v],\] where $[v]$ is a Lie monomial whose bracketing is the same as the bracketing of $[x_{i_1},\cdots,x_{i_q}]$. So 
\[f\left([x_{i_1},\cdot,x_{i_q}]\right)=\sum_{\mathcal{I}_u\cdot w\in\mathcal{O}_u}\left(a_{\mathcal{I}_u\cdot w}\cdot\sum_{v\in B_{n,q}, \mathcal{I}_u\cdot v=\mathcal{I}_u\cdot w}[v]\right).\]

For $u\in B_{n,q}$ there exists a unique $u'\in S_{n,q}$ such that $u=\sigma(u')$ for some $\sigma \in \Sigma_q$. Let $R_{n,q}=\left(S_{n,q}\cup \{u\}\right)\setminus\{u'\}$. For every $(a_{\mathrm{I}_u\cdot w})_{\mathrm{I}_u\cdot w\in\mathcal{O}_u}\in K^{d_u}$, define $g:R_{n,q}\to V_n^{\otimes q}$ as follows:
\begin{equation*}
g(r) =
\begin{cases*}
\sum_{\mathcal{I}_u\cdot w\in\mathcal{O}_u}\left(a_{\mathcal{I}_u\cdot w}\cdot\sum_{v\in B_{n,q}, \mathcal{I}_u\cdot v=\mathcal{I}_u\cdot w}v\right) & if $r=u$, \\
  0    & if $r\neq u$.
\end{cases*}
\end{equation*}
By Corollary \ref{cor 4.1.4}, $g$ can be uniquely extended to some $f\in \mathrm{Schur}(q,n,K)$. So the coefficients of $(a_{\mathrm{I}_u\cdot w})_{\mathrm{I}_u\cdot w\in\mathcal{O}_u}$ can be any element in $K^{d_u}$ and this completes the proof.
\end{proof}

\begin{pro}\label{4.1.11}
The map $\mathfrak{b}_n$ commutes with Schur algebra elements. That is, for all $q\geq 1$ and $f\in\mathrm{Schur}(q,n,K)$,
\[f\circ \mathfrak{b}_{n,2q}=\mathfrak{b}_{n,2q}\circ f.\]
\end{pro}
\begin{proof}
Similar to the proof of Theorem \ref{4.1.10}, we obtain 
\[f\circ\mathfrak{b}_{n,2q}(x_{i_1}\otimes\cdots \otimes x_{i_q})=\mathfrak{b}_{n,2q}\circ f(x_{i_1}\otimes\cdots\otimes x_{i_q})\]
for all $x_{i_1}\otimes\cdots\otimes x_{i_q}\in V_n(K)^{\otimes q}$.
\end{proof}

\begin{pro}\label{4.1.12}
	The Schur algebra elements send Lie elements to Lie elements. That is, for all $f\in\mathrm{Schur}(q,n,K)$ and $u\in L_K^q(X_n)$,
	\[f(u)\in L_K^q(X_n).\]
\end{pro}
\begin{proof}
Consider the simple Lie monomial \[v=[x_{i_1},\ldots,x_{i_q}]\in V_n(K)^{\otimes q}.\]By the definition of $\mathfrak{b}_{n,2q}$,
\[\mathfrak{b}_{n,2q}\left(f(x_{i_1}\otimes\cdots\otimes x_{i_q})\right)\in L_K^q(X_n).\] That is, 
\[f\left(\mathfrak{b}_{n,2q}(x_{i_1}\otimes\cdots\otimes x_{i_q})\right)\in L_K^q(X_n).\] Since every Lie element can be written as a linear combination of simple Lie monomials, the result follows. 
\end{proof}

Now we discuss the action of the Schur Algebra on the Lie algebra of derivations of a free Lie algebra.

\begin{df}
	For any map $h:X_n\to L^q(X_n)$, the unique $g\in\mathrm{Der}_q\left(L(X_n)\right)$ satisfying $h=g\vert_{X_n}$ is called the extension of $h$ and is denoted by \[g=h^{\mathrm{Ext}}\]
\end{df}

\begin{df}\label{1.11}
	The action of $f\in\mathrm{Schur}(q,n,K)$ on $\mathrm{Der}_q\left(L_K(X_n)\right)$ is defined as the map
	\[\Phi_f:\mathrm{Der}_q\left(L_K(X_n)\right)\to\mathrm{Der}_q\left(L_K(X_n)\right)\]with
	\[\Phi_f(\theta)=\left(f\circ I_{n,q}\circ\theta\vert_{X_n}\right)^{\mathrm{Ext}}.\] 
\end{df}

\begin{pro}\label{pro 1.1.4}
The action of the Schur algebra $\mathrm{Schur}(q,n,K)$ on $V_n^{\otimes q}$ induces an action on $\mathrm{Der}_q\left(L_K(X_n)\right)$ and therefore an action on $\mathrm{Der}_q\left(L_K(X_n)\right)$
\begin{align*}
\Phi_{n,q}:\mathrm{Schur}(q,n,K)\times\mathrm{Der}_q(L(X_n))&\to\mathrm{Der}_q(L(X_n))\\
(f,\theta)&\mapsto \Phi_f(\theta).
\end{align*}
\end{pro}
\begin{proof}
By the definition of $\mathrm{Schur}(q,n,K)$, $V_n^{\otimes q}$ is a left $\mathrm{Schur}(q,n,K)$ module. Since every Schur algebra element preserves Lie elements, the result follows.
\end{proof}

%\begin{pro}
%By the definition of the Schur algebra $\mathrm{Schur}(q,n,K)$, $V_n^{\otimes q}$ is a left $\mathrm{Schur}(q,n,K)$-module. The conclusion of the theorem follows from the fact that every Schur algebra elements preserves Lie elements.
%\end{pro}
The following theorem can be derived directly from Proposition \ref{pro 1.1.4}.
\begin{thm}\label{thm 1.1.7}
	The Lie algebra of derivations $\mathrm{Der}_*(L_K(X_n))$ and  $Der_*^{\Delta}(L_K(X_n))$ are graded left-modules over the graded Schur algebra $Sch(n,K)$.  The action of the Schur algebra $Sch(n,K)$ on $\mathrm{Der}_*(L_K(X_n))$ and  $Der_*^{\Delta}(L_K(X_n))$ is given by Proposition \ref{pro 1.1.4}.
\end{thm}

\section{The Associativity and Anti-commutativity of the Schur Algebra}
Here we study some properties of Schur algebras and their action on
the Lie algebra of derivations of a free Lie algebra. By using the inter-degree multiplication
induced from a transfer-like map, it is shown that the Schur algebra forms an associative and anti-commutative graded K-algebra, and the so-called Schur operad is also constructed. \label{Associativity} The proof to the following lemma is clear:
\begin{lem}\label{lem 4.5.1}
	Given groups $A\leq B\leq G$, let $L_{A,B}$ and $L_{B,G}$ be complete sets of representatives of the left cosets of $A$ in $B$ and of $B$ in $G$, respectively. Then the following hold:
	\begin{enumerate}
		\item If $(g,h)$ and $(g',h')$ are distinct elements in $L_{B,G}\times L_{A,B}$, then $gh\neq g'h'$.
		\item The set \[\{gh\mid(gh)\in L_{B,G}\times L_{A,B}\}\] is a complete set of representatives of the left cosets of $A$ in $G$.
	\end{enumerate}
\end{lem}
Let $\lambda$ be a partition of a positive integer $d$; $d=a_1+\cdots+a_k$ (the order matters). We define the \textit{Young subgroup} of $\Sigma_d$, denoted by $\Sigma_\lambda$ as the direct product $\Sigma_{a_1}\times\cdots\times\Sigma_{a_k}$. Let $L_{\lambda}$ a complete set of representatives of the left cosets of $\Sigma_\lambda$ in $\Sigma_d$. Define the ``transfer-like'' map 
\[M_{\lambda}:\mathrm{Schur}(a_1,n,K)\otimes\cdots\otimes\mathrm{Schur}(a_k,n,K)\to \mathrm{Schur}(d,n,K)\] by
\[\left(M_{\lambda}(f_1\otimes \cdots\otimes f_k)\right)(x_{i_1}\otimes\cdots\otimes x_{i_d})=\sum_{\sigma\in L_{\lambda}}\sigma\left((f_1\otimes \cdots\otimes f_k)\left(\sigma^{-1}(x_{i_1}\otimes\cdots\otimes x_{i_d})\right)\right).\]
One easily checks that $M_{\lambda}$ is well-defined and is independent of the choice of $L_{\lambda}$.
% {(\color{red}Does this need a proof?)}

We define the inter-degree product \[\boxtimes:\mathrm{Schur}(0,n,K)\times\mathrm{Schur}(q,n,K)\to\mathrm{Schur}(q,n,K)\] by $f\boxtimes g:=f\cdot g$, where the right hand side makes sense as $f$ is in $\mathrm{Schur}(0,n,K)=K$. 

%\begin{pro}\label{thm 1.1.8}
%	The inter-degree product $\boxtimes$ is commutative, associative, and distributive (both left and right) over addition.
%\end{pro}
%The following three lemmas will give the proof of the proposition.
To simplify the notation, for $\lambda=(a_1,\ldots,a_k)$, when specification of the partition is unnecessary, we denote $M_{\lambda}(f_1\otimes\cdots\otimes f_k)$ by $\displaystyle\star_{i=1}^{k}f_i$ and when $k=2$, we also write it as $f_1\star f_2$. We shall show this product is associative and commutative. 
\begin{lem}\label{prop 4.5.2}
For $f_i\in\mathrm{Schur}(a_i,n,K)$ $i=1,2,3$,
\[\star_{i=1}^{3}f_i=(f_1\star f_2)\star f_3=f_1\star(f_2\star f_3).\]	
\end{lem}
\begin{proof}
Let $\lambda'$ be the partition $(a_1,a_2)$ of $a_1+a_2$ and let $L_{\lambda'}$ be a complete set of representatives of the left cosets of $\Sigma_{a_1}\times\Sigma_{a_2}$ in $\Sigma_{a_1+a_2}$. The natural embedding of groups $\Sigma_{a_1}\times\Sigma_{a_2}\hookrightarrow \Sigma_{a_1+a_2}$ gives the embedding
	\[\Sigma_{a_1}\times\Sigma_{a_2}\times\Sigma_{a_3}\hookrightarrow \Sigma_{a_1+a_2}\times\Sigma_{a_3}.\] So $L_{\lambda'}$ can also be regarded as a complete set of representatives of the left cosets of $\Sigma_{a_1}\times\Sigma_{a_2}\times\Sigma_{a_3}$ in $\Sigma_{a_1+a_2}\times\Sigma_{a_3}$. Consider the groups
	\[\Sigma_{a_1}\times\Sigma_{a_2}\times\Sigma_{a_3}\leq\Sigma_{a_1+a_2}\times\Sigma_{a_3}\leq\Sigma_{a_1+a_2+a_3}. \]Let $L_{\alpha}$ be a complete set of representatives of the left cosets of $\Sigma_{a_1+a_2}\times\Sigma_{a_3}$ in $\Sigma_{a_1+a_2+a_3}$. So by Lemma~\ref{lem 4.5.1}, the set
	\[L_{\beta}:=\{\tau\sigma\mid(\tau,\sigma)\in L_{\alpha}\times L_{\lambda'}\}\] is a complete set of representatives of the left cosets of $\Sigma_{a_1}\times\Sigma_{a_2}\times\Sigma_{a_3}$ in $\Sigma_{a_1+a_2+a_3}$. From the definitions of $M_{\lambda}$, we conclude that
\[\star_{i=1}^{3}f_i=f_1\star(f_2\star f_3).\] The proof to \[\star_{i=1}^{3}f_i=(f_1\star f_2)\star f_3\]	is similar.

\end{proof}
\begin{lem}\label{prop 4.5.3}
For any partition $\lambda=(a_1,a_2)$ of a positive integer $d$,
\[M_{\lambda}(f_1\otimes f_2)=M_{\lambda'}(f_2\otimes f_1),\]  
where $\lambda'$ is the partition $(a_2,a_1)$ of $d$.
\end{lem}
\begin{proof}
	Write $x_{i_1}\otimes\cdots\otimes x_{i_n}$ as the tensor product 
	\[x_{i_1}\otimes\cdots\otimes x_{i_n}=b_1\otimes b_2,\]
	where $b_i\in V_n^{\otimes a_i}$ for $i=1$ or $2$. Let $\tau =(12)\in\Sigma_2$ and let $\sigma_{\tau}$ be the unique element in $\Sigma_d$ such that 
	\[\sigma_{\tau}(x_{i_1}\otimes\cdots\otimes x_{i_n})=b_2\otimes b_1.\] 
	A direct calculation gives
	\[f_1\otimes f_2=\sigma_{\tau}^{-1}\circ (f_2\otimes f_1)\circ\sigma_{\tau},\]which implies
	\[M_{\lambda}(f_1\otimes f_2)=\sum_{\sigma\in L_{\lambda}}(\sigma\sigma_{\tau}^{-1})\circ(f_2\otimes f_1)\circ(\sigma\sigma_{\tau}^{-1})^{-1}.\]
	One easily shows that the right hand side of the equation above is precisely 
	\[M_{\lambda'}(f_2\otimes f_1).\]
\end{proof}
The same ideas can be used to prove the case when the partition of $d$ is $\lambda=(a_1,\ldots,a_k)$. More precisely, we have the following proposition:
\begin{pro}\label{prop 4.5.3}
Let $\lambda=(a_1,\ldots,a_k)$ be a partition of a positive integer $d$ and let $\tau\in \Sigma_k$. Let $\lambda_{\tau}$ be the partition $(a_{\tau(1)},\ldots,a_{\tau(k)})$. Then
\[M_{\lambda}(f_1\otimes \cdots f_k)=M_{\lambda_{\tau}}(f_{\tau(1)}\otimes \cdots f_{\tau(k)}).\] 
\end{pro}
The following lemma follows directly from the definition of $M_{\lambda}$.
\begin{lem}\label{p80}
	If $f,f_1,f_2\in\mathrm{Schur}(0,n,K)=K$ and $g, g_1,g_2\in\mathrm{Schur}(q,n,K)$, 
	\[f\boxtimes(g_1+g_2)=f\boxtimes g_1+f\boxtimes g_2\]
	and
	\[(f_1+f_2)\boxtimes g=f_1\boxtimes g+f_2\boxtimes g.\]
\end{lem}

Putting these together, we have the following
\begin{thm}\label{thm 1.1.8 and 1.1.9}
The inter-degree multiplication $\boxtimes$ in the Schur algebra $\mathrm{Schur}(n;K)$
is associative and commutative and satisfies the distributive laws and thus the Schur algebra $\mathrm{Schur}(n;K)$ is an associative and anti-commutative
graded $K$-algebra.
\end{thm}
\section{Action of the Schur Algebra on Derivations and Operads}\label{Action}

In this section, the Schur-operad structure is constructed.  We show that the Lie algebra of derivations can be generated by quadratic derivations together with the action of the Schur operad.

Let $\overline{S}(\widetilde{M}_n,K)$ denote the smallest sub-Lie algebra of $\mathrm{Der}_*L_K(X_n)$ which contains $\widetilde{M}_n$ and is invariant under the map $\Phi_f$ for any $f\in\mathrm{Schur}(q,n,K)$ for some $q\geq 2$, where $\Phi_f$ is defined in Definition \ref{1.11}. Let 
\[\overline{S}_q(\widetilde{M}_n,K)=\overline{S}(\widetilde{M}_n,K)\bigcap\mathrm{Der}_q(L_K(X_n)).\]

\begin{pro}\label{prop 4.2.1}
Suppose $n,p\geq 2$. Then $\overline{S}_q(\widetilde{M}_n,K)=\mathrm{Der}_q(L_K(X_n))$ if and only if the set
\[\mathcal{B}(p,n,K)=\{f_{i,[x_j,u]}\in\mathrm{Der}_q(L_K(X_n))\mid 1\leq i,j\leq u\}\subseteq\overline{S}_q({M_n},K),\] where $u$ is a simple Lie monomial of degree $2(p-1)$.	
\end{pro}
\begin{proof}
We only prove if $\mathcal{B}(p,n,K)\subseteq\overline{S}_q({M_n},K)$, then $\overline{S}_q({M_n},K)=\mathrm{Der}_q(L_K(X_n))$ as the converse is trivial. 	
	Define \[f_{i,w}(x_s)=\begin{cases}
	0,\text{if } s\neq i,\\w,\text{if } s=i.
	\end{cases}\] It is clear that for any $p,n\geq 2$, 
	\[\{f_{i,w}\in\mathrm{Der}_p\left(L_K(X_n)\right)\mid 1\leq i\leq n, w\in L_K^p(X_n)\}\] is a generating set of $\mathrm{Der}_q(L_K(X_n))$. Since every Lie monomial and be written as a linear combination of simple Lie monomials, $\mathcal{B}(p,n,K)$ is a generating set of $\mathrm{Der}_q(L_K(X_n))$. Thus $\mathrm{Der}_q(L_K(X_n))\subseteq \overline{S}_q({M_n},K)$.
\end{proof}

A direct computation gives the following proposition:
\begin{pro}\label{prop 4.2.2}
	For any $i,j$ in $\{1,\ldots,n\}$ with $i\neq j$ and any Lie monomial $u$ of degree $2k$,'
	 \[[{\chi_{i,j}},f_{j,u}]=-f_{i,[x_i,u]}+f_{j,{\chi_{i,j}}(u)},\]where $f_{j,u}$, $f_{i,[x_i,u]}$ and $f_{j,{\chi_{i,j}}(u)}$ are defined in the proof of Proposition \ref{prop 4.2.1}.
\end{pro}

Let $u$ be a Lie monomial of degree $2k$, then ${\chi_{i,j}}(u)$ is equal to the sum of Lie monomials where each monomial is obtained from $u$ by replacing each appearance of $x_i$ in $u$ by $[x_i,x_j]$.
\begin{ex}We assume that $i\neq k$.
\begin{align*}
{\chi_{i,j}}\left(\left[[x_i,x_k],x_i\right]\right)&=\left[{\chi_{i,j}}\left([x_i,x_k]\right),x_i\right]+\left[[x_i,x_k],{\chi_{i,j}}(x_i)\right]\\
&=\left[\left[{\chi_{i,j}}(x_i),x_k\right],x_i\right]+\left[\left[x_i,{\chi_{i,j}}(x_k)\right],x_i\right]+\left[[x_i,x_k],[x_i,x_k]\right]\\
&=\left[\left[[x_i,x_k],x_k\right],x_i\right]+\left[\left[x_i,0\right],x_i\right]+\left[[x_i,x_k],[x_i,x_k]\right]\\
&=\left[\left[[x_i,x_k],x_k\right],x_i\right]+\left[[x_i,x_k],[x_i,x_k]\right]
\end{align*}
\end{ex}
The proof of the following lemma is found on pages 65-67 of \cite{Jin}.
\begin{lem}\label{lemma 4.2.5}For every Lie monomial $u$ of degree $2k$ ($k\geq 2$), there exists $h\in\mathrm{Schur}(k+1,n,K)$ such that 
\[h\left({\chi_{i,j}}(u)\right)=0\] and 
\[h\left(-[x_i,u]=[x_i,u]\right),\] which implies
\[\Phi_h\left([{\chi_{i,j}},f_{j,u}]\right)=f_{i,[x_i,u]},\]
where $\Phi_h$ is defined as in Definition \ref{1.11}.	
\end{lem}
\begin{thm}\label{4.2.7}
For every $n$, $\overline{S}({M_n},K)=\mathrm{Der}_*\left(L_K(X_n)\right)$.
\end{thm}
\begin{proof}
The case $n=1$ is trivial and we assume $n\geq 2$ and it suffices to prove
\[\overline{S}_p({M_n},K)=\mathrm{Der}_p\left(L_K(X_n)\right)\] and we shall prove by induction on $p$. Since ${M_n}$ is a generating set of $\mathrm{Der}_2\left(L_K(X_n)\right)$, the result is true for $p=2$. Assume \[\overline{S}_p({M_n},K)=\mathrm{Der}_p\left(L_K(X_n)\right)\] for $p=2,\ldots,k$.

Given any $1\leq i\leq n$ and any Lie monomial $u$ of degree $2k$. Let $1\leq j\leq n$ and $j\neq i$. Then by Prop \ref{prop 4.2.2}, \[[{\chi_{i,j}},f_{j,u}]=-f_{i,[x_i,u]}+f_{j,{\chi_{i,j}}(u)}.\] By Lemma \ref{lemma 4.2.5}, there exists $h\in\mathrm{Schur}(k+1,n,K)$ such that \[\Phi_h\left([{\chi_{i,j}},f_{j,u}]\right)=f_{i,[x_i,u]},\] which implies $f_{i,[x_i,u]}\in \overline{S}_{k+1}({M_n},K)$ for any $i$ and any Lie monomial $u$ of degree $2k$.

For a given $\sigma\in \Sigma_n$ and $t\geq 1$, let $\zeta_{t,\sigma}$ be the linear map $V_n^{\otimes t}\to V_n^{\otimes t}$ defined by 
\[\zeta_{t,\sigma}(x_{j_1}\otimes\cdots\otimes x_{x_{j_t}})=x_{\sigma(j_1)}\otimes \cdots\otimes x_{\sigma(j_t)}.\]

A direct calculation shows that
\[\zeta_{k+1,(i,j)}\left([x_i,\zeta_{k,(i,j)}(u)]\right)=[x_j,u].\]Therefore, 
\[\Phi_{\zeta_{k+1,(i,j)}}\left(f_{i,[x_i,v]}\right)=f_{i,[x_j,u]};\]
that is, $f_{i,[x_j,u]}\in\overline{S}_{k+1}({M_n},K) $ for any $1\leq i\neq j\leq n$ and any Lie monomial $u$ of degree $2k$ and by Proposition \ref{prop 4.2.1} the proof is finished.
\end{proof}

Let ${v}_1,\ldots,{v}_n$ be elements in ${M_n}$ such that either each ${v}_i$ is ${\chi_{i,j}}$ with $i\neq j$ or each ${v}_i$ is ${\theta}_{i,[x_s,x_t]}$ for distinct $i,s$ and $t$. Then it is easy to check that every element in ${M_n}$ can be obtained from ${\gamma_n}:=\{{v}_1,\ldots,{v}_n\}$ via the action of $\mathrm{Schur}(2,n,K)$. Thus we have the following corollary:
\begin{cor}\label{4.2.8}
For every $n$, $\overline{S}({\Gamma_n},K)=\mathrm{Der}_*\left(L_K(X_n)\right)$.
\end{cor}

%\subsection{Action of Diagonal Sub-Algebra $\mathrm{D}(n,K)$ on $\mathrm{Der}\left(L_K(X_n)\right)$}
%WE NEED TO DEFINE A FEW THINGS HERE
%\begin{thm}\label{4.4.1}
%For $n\neq 2$, we have	$\overline{S}^D\left(\widetilde{M}_n,K\right)=\mathrm{Der}_\ast\left(L_K(X_n)\right)$.
%\end{thm}
%The proof is similar to the proof of Theorem \ref{4.2.7} and thus omitted.

%\section{The Schur Operad}\label{Operad}
\begin{df}
An operad consists of
\begin{enumerate}
	\item a sequence $\left(P(m)\right)_{m\in\mathbb{Z}_+}$ of sets whose elements $\theta$ are called the $m$-ary operation of $P$.
	\item an element $1\in P(1)$ called the identity.
	\item for all positive integers $m,k_1,\ldots,k_m$, a composition function
	\begin{align*}
	\circ: P(m)\times P(k_1)\times\cdots\times P(k_m)&\to P(k_1+\cdots k_m)\\
	(\theta,\theta_1,\ldots, \theta_m)&\mapsto \theta\circ(\theta_1,\ldots,\theta_m)
	\end{align*}
\end{enumerate}
satisfying the following coherence axioms:
\begin{itemize}
	\item identity: $\theta\circ(1,\ldots,1)=\theta=1\circ\theta$.
	\item associativity: 
	\begin{align*}
	& \theta \circ (\theta_1 \circ (\theta_{1,1}, \ldots, \theta_{1,k_1}), \ldots, \theta_n \circ (\theta_{n,1}, \ldots,\theta_{n,k_n})) \\
	= {} & (\theta \circ (\theta_1, \ldots, \theta_n)) \circ (\theta_{1,1}, \ldots, \theta_{1,k_1}, \ldots, \theta_{n,1}, \ldots, \theta_{n,k_n}),
	\end{align*}
	where the number of arguments correspond to the arities of the operations.
\end{itemize}
\end{df}
A morphism $f:P\to Q$ of operads is naturally defined a sequence \[\left(f_m:P(m)\to Q(m)\right)_{m\in\mathbb{Z}_+}\] of maps
	which is compatible with identity and composition.

The following theorem, whose proof is obvious, constructs the operad of Schur Algebras. 
\begin{pro}\label{1.1.10}
Let $\mathit{Schur}=\left(\mathrm{Schur}(m-1,n,K)\right)_{m\in\mathbb{Z}_+}$. For each $m,k_1,\ldots,k_m\in\mathbb{Z}_{+}$, define the operad composition by 
	\begin{align*}
\circ: P(m)\times P(k_1)\times P(k_m)&\to P(k_1+\cdots k_m)\\
(\theta,\theta_1,\ldots, \theta_m)&\mapsto \theta\boxtimes\theta_1\boxtimes\cdots\boxtimes\theta_m
\end{align*}
and the identity of $\mathscr{S}chur$ is defined as the identity element of $\mathrm{Schur}(0,n,K)=K=P(1)$. This defines $\mathit{Schur}$ as an operad.
\end{pro}

Proposition \ref{pro 1.1.4} and Corollary \ref{4.2.8} imply the following result:
\begin{thm}\label{1.1.13}
The Schur Operad $\mathscr{S}ch$ acts on $T\left(V_n(K)\right)$, $L_K(X_n)$, $\mathrm{Der}_*^{\triangle}\left(L_K(X_n)\right)$ and $\mathrm{Der}_*\left(L_K(X_n)\right)$. Furthermore, $\mathrm{Der}_*\left(L_K(X_n)\right)$ is generated by  $\{{v}_1,\ldots,{v}_n\} \subseteq {M_n}$ where either each ${v}_i$ is ${\chi_{i,j}}$ with $i\neq j$ or each ${v}_i$ is ${\theta}_{i,[x_s,x_t]}$ for distinct $i,s$ and $t$. 
\end{thm}

\section{A Short Exact Sequence of Lie Algebras}\label{Exact}
We introduce a short exact sequence of Lie algebras which is obtained from certain group filtrations.

Define the group epimorphism $p_n$ by $p_n:P\Sigma_n\to P\Sigma_{n-1}$ by
\[p_n(\chi_{i,j})=\begin{cases}
\chi_{i,j}\text{, if }n\in\{i,j\}\\1\text{, otherwise.}
\end{cases}\]
Then there is a short exact sequence of groups \cite{CPVW}
\[1\to\mathfrak{K}_n\to P\Sigma_n\xrightarrow{p_n} P\Sigma_{n-1}\to 1.\]
The kernel $\mathfrak{K}_n$ is generated by the set $P_n:=\{\chi_{i,j}\mid n\in\{i,j\}\}$ in the group $P\Sigma_n$ \cite {CPVW}. 

Define $q_n:\mathfrak{K}_n\to\mathbb{Z}^{n-1}$ defined by
\[q_n(\chi_{i,n})=\chi_{i,n}\]
and 
\[q_n(w)=1\text{, if } w \text{ is a product of }\chi_{i,j}\text{'s other than }\chi_{i,n}\text{'s}. \]
Then there is a short exact sequence \cite{CPVW}
\[1\to\mathfrak{W}_n\to\mathfrak{K}_n\xrightarrow{q_n}\mathbb{Z}_{n-1}\to 1.\]

\begin{lem}\label{5.2.2}
	Under the notation of the short exact sequence above, we have
	\[\mathrm{gr}_2^J(\mathbb{Z}^{n-1})=1,\]
	where $\mathbb{Z}^{n-1}$ is the free abelian group generated by $\{\chi_{i,n}\mid 1\leq i<n\}$.
\end{lem}
Recall for $h\in J^mIA_n$, $\widetilde{J}_{n,m}(h\cdot J^{m+1}IA_n)$ is the unique derivation of degree $2(m+1)$ given by the conditions
\begin{equation}\label{2.17}
\widetilde{J}_{n,m}\left(h\cdot J^{m+1}IA_n\right)(x_i)=\phi_{n,m+1}^{-1}\left(w_{i,m+1\cdot\Gamma^{m+2}F(X_n)}\right)
\end{equation}
for $i=1,\ldots,n$; see page 29 of \cite{Jin} for details. 
Recall also the morphism of Lie algebras
\[\psi_{n,\ast}:\mathrm{gr}_\ast^\mathrm{LCS}(IA_n)\to\mathrm{gr}_\ast^\mathrm{J}(IA_n)\] induced by the quotient map
\[IA_n^m/IA_n^{m+1}\to J^mIA_n/J^{m+1}IA_n.\] Composing with $\widetilde{J}_{n,\ast}$, we obtain the morphism of Lie algebras
\begin{equation}\label{2.22}\widetilde{J}_{n,\ast}\circ \psi_{n,\ast}:\mathrm{gr}_\ast^\mathrm{LCS}(IA_n)\to\mathrm{Der}_\ast\left(L_\mathbb{Z}(X_n)\right).
\end{equation}
\begin{proof}
We write every non-trivial element in $\mathbb{Z}^{n-1}$ as 
\[\chi_{i_1,n}^{e_1}\cdot\chi_{i_2,n}^{e_2}\cdots\chi_{i_r,n}^{e_r}\]for some $r\geq 1$, $1\leq i_1< i_2<\cdots i_r<n$ and $e_1,e_2,\ldots,e_r\neq 0$.
By the formula (\ref{2.17}), one gets
\[\chi_{i_1,n}^{e_1}\cdot\chi_{i_2,n}^{e_2}\cdots\chi_{i_r,n}^{e_r}(x_{i_1})=[x_{i_1},x_{j_1}]^{e_1}\cdot u,\] where $u\in\Gamma^3F(X_n)$. Thus 
\[\mathbb{Z}^{n-1}\cap J^2IA_n=1\] and the proof is completed.
\end{proof}
Now we prove a general version of the Falk-Randell Theorem \cite{Falk}.
\begin{thm}\label{1.1.15}
Let 
\[1\to A\to B\to C\to 1\] be a split short exact sequence of groups thus allowing $A$ and $C$ be identified as subgroups of $B$.  let $\{B_m\}_{m\geq 1}$ be a filtration of $B$. Let $A_m=B_m\cap A$ and $C_m=B_m\cap C$ and therefore, $\{A_m\}$ and $\{C_m\}$ are filtrations of $A$ and $C$, respectively. If $B_m=A_mC_m=\{ac\mid a\in A_m\text{ and }c\in C_m\}$ for each $m$, then there is a short exact sequence of Lie algebras
\[0\to\mathrm{gr}(A)\to\mathrm{gr}(B)\to\mathrm{gr}(C)\to 0\] which splits as a sequence of abelian groups.
\end{thm}
\begin{proof}
	For each $m$, the split short exact sequence
	\[1\to A_m\to B_m\to C_m\to 1\]	gives the following commutative diagram, the rows of which splits.
	\begin{center}

	\begin{tikzcd}
	&1\arrow[d] & 1 \arrow[d] & 1 \arrow[d]& \\1 \arrow[r]&	A_{m+1} \arrow[r] \arrow[d] & B_{m+1} \arrow[r] \arrow[d] & C_{m+1} \arrow[r]\arrow[d] &	1 \\1 \arrow[r]&	A_m \arrow[r] \arrow[d] & B_m \arrow[r] \arrow[d]  &C_m \arrow[r]\arrow[d] &	1  \\0 \arrow[r]&
	\mathrm{gr}_m(A) \arrow[r] \arrow[d] & \mathrm{gr}_m(B) \arrow[r] \arrow[d]  &\mathrm{gr}_m(C)\arrow[r]\arrow[d] &	0\\	&1         & 1      &  1&
	\end{tikzcd}
		\end{center}
\end{proof}
\begin{thm}\label{1.1.15}
	Let 
	\[1\to A\to B\to C\to 1\]
	be a split short exact sequence of groups and let $\{B_m\}_{m\geq 1}$ be a filtration of $B$. We identify $A$ and $C$ as subgroups of
	$B$ according to the corresponding embedding and cross section. Let $A_m=B_m\cap A$ and $C_m=B_m\cap C$ and therefore, $\{A_m\}$ and $\{C_m\}$ are filtrations of $A$ and $C$, respectively. Then there is a short exact sequence of Johnson Lie algebras
	\[0\to\mathrm{gr}_\ast^J(\mathfrak{W}_n)\to\mathrm{gr}_\ast^J(\mathfrak{K}_n)\to\mathbb{Z}^{n-1}\to 0\] which splits as a sequence as abelian groups and therefore,
	\[\mathrm{gr}_\ast^J(\mathfrak{K}_n)\cong\mathbb{Z}^{n-1}\oplus\mathrm{gr}_\ast^J(\mathfrak{W}_n).\]
\end{thm}
\begin{proof}
Since the group $\mathfrak{K}_n$ is generated by the $w^{-1}\chi_{i,j}w$'s, where $\chi_{i,j}\in P_n$, $w$ is a (possibly empty) product of $\chi_{i',j'}'s$. Thus every element $a\in\mathfrak{K}_n$ can be written as 
\[a=\chi_{i_1,j_1}^{e_1}\chi_{i_2,j_2}^{e^2}\cdots\chi_{i_r,j_r}^{e_r}\cdot u\] for $e_k\in\mathbb{Z}$, $k=1,\ldots, r$ and some $u\in\Gamma^2\mathfrak{K}_n$.
Then 
\begin{align*}
q_n(a)&=\sum_{k=1}^{r}e_k\cdot q_n(\chi_{i_k,j_k})\\
&=\sum_{1\leq k\leq r, j_k=n}e_k\cdot \chi_{i_k,j_k}.
\end{align*}
Suppose $a\in\mathfrak{W}_n$ and $c=\displaystyle\sum_{l=1}^{n-1}d_i\chi_{l,n}$ is a non-trivial element in $\mathbb{Z}^{n-1}$. Then
\[ac=\chi_{i_1,j_1}^{e_1}\chi_{i_2,j_2}^{e^2}\cdots\chi_{i_r,j_r}^{e_r}\cdot u\cdot\chi_{1,n}^{d_1}\chi_{2,n}^{d_2}\cdots\chi_{n-1,n}^{d_{n-1}}.\] 
Since $u\in\Gamma^2\mathfrak{K}_n\subset J^2\mathfrak{K}_n$,
\[ac\dot J^2\mathfrak{K}_n=\chi_{i_1,j_1}^{e_1}\chi_{i_2,j_2}^{e^2}\cdots\chi_{i_r,j_r}^{e_r}\chi_{1,n}^{d_1}\chi_{2,n}^{d_2}\cdots\chi_{n-1,n}^{d_{n-1}}J^2\mathfrak{K}_n.\]
A straightforward computation gives 
\[\widetilde{J}_{1,\mathfrak{K}_n}(ac\cdot J^2\mathfrak{K}_n)=\sum_{1\leq k\leq r, j_k\neq n}e_k\cdot\widetilde{\chi}_{i_k,j_k}+\sum_{l=1}^{n-1}d_l\cdot\widetilde{\chi}_{l,n}\neq 0\]
since not all $d_l$'s are zeros. That is, for every $a\in\mathfrak{W}_n$ and $c\neq 1$, $ac\notin J^2\mathfrak{K}_n$. By Lemma \ref{5.2.2} and Theorem \ref{1.1.15}, there is a short exact sequence of Johnson Lie algebras
\[0\to\mathrm{gr}_\ast^J(\mathfrak{W}_n)\to\mathrm{gr}_\ast^J(\mathfrak{K}_n)\to\mathbb{Z}^{n-1}\to 0\] which splits as a sequence as abelian groups and therefore,
\[\mathrm{gr}_\ast^J(\mathfrak{K}_n)\cong\mathbb{Z}^{n-1}\oplus\mathrm{gr}_\ast^J(\mathfrak{W}_n).\] 

\section*{Acknowledgement} 
Mohamed Elhamdadi was partially supported by Simons Foundation collaboration grant 712462.

\subsection*{In Memoriam} Frederick Cohen passed away on january 16, 2022 at the age of 76 and we all initiated the present work and Fred fully contributed to the whole paper. 

\end{proof}
  
\end{document}